\documentclass[a4paper,reqno]{amsart}
\usepackage{amscd,amsfonts,amssymb,dsfont}
\usepackage{textgreek}
\usepackage{tikz}
\usepackage{pgfplots}
\usetikzlibrary{automata,positioning,calc,trees}
\usetikzlibrary{intersections,pgfplots.fillbetween}
\usepackage{graphicx,tikz}
\usepackage{pgf,tikz,pgfplots}
\usetikzlibrary{arrows}
\usepackage{enumerate}
\usepackage[shortlabels]{enumitem}
\usepackage{mathrsfs}
\usepackage{amssymb,amsmath,amsthm,color}
\usepackage{caption,subcaption}
\usepackage[initials,nobysame,alphabetic,bibtex-style]{amsrefs}
\usepackage{hyperref}
\usepackage{thmtools}
\usepackage{cleveref}
\usepackage{url}
\usepackage{setspace}
\usepackage{float}
\usepackage{makecell}
\tikzstyle{startstop} = [rectangle, rounded corners, minimum width=3cm, minimum height=1cm,text centered, draw=black]
\tikzstyle{process} = [rectangle, rounded corners, minimum width=3cm, minimum height=1cm, text centered, draw=black]
\tikzstyle{arrow} = [thick,->,>=stealth]

\BibSpec{article}{%
	+{}  {\PrintAuthors}                {author}
	+{,} { \textit}                     {title}
	+{.} { }                            {part}
	+{:} { \textit}                     {subtitle}
	+{,} { \PrintContributions}         {contribution}
	+{.} { \PrintPartials}              {partial}
	+{,} { }                            {journal}
	+{}  { \textbf}                     {volume}
	+{}  { \PrintDatePV}                {date}
	+{,} { \issuetext}                  {number}
	+{,} { \eprintpages}                {pages}
	+{,} { }                            {status}
	+{,} { \url}                        {url}    
	+{,} { \PrintDOI}                   {doi}
	+{,} { available at \eprint}        {eprint}
	+{}  { \parenthesize}               {language}
	+{}  { \PrintTranslation}           {translation}
	+{;} { \PrintReprint}               {reprint}
	+{.} { }                            {note}
	+{.} {}                             {transition}
}

\date{\today}
\allowdisplaybreaks
\newcommand{\R}{{\mathbb R}}       
\newcommand{\N}{{\mathbb N}}
\newcommand{\Sp}{{\mathbb S}}
\newcommand{\Z}{{\mathbb Z}}    
\newcommand{\supp}[1]{\operatorname{supp}(#1)}   

\newtheorem{theorem}{Theorem}[section]
\newtheorem{lemma}[theorem]{Lemma}

\newtheorem*{theorem*}{Theorem}

\numberwithin{equation}{section}

\begin{document}
	
\title[Endpoint variation and jump inequalities]{Endpoint variation and jump inequalities for rough singular integrals}
\author[Bhojak]{Ankit Bhojak}
	\address{Ankit Bhojak\\
		Department of Mathematics\\
		Indian Institute of Science Education and Research Bhopal\\
		Bhopal-462066, India.}
	\email{ankitb@iiserb.ac.in}

\author[Shrivastava]{Saurabh Shrivastava}
	\address{Saurabh Shrivastava\\
		Department of Mathematics\\
		Indian Institute of Science Education and Research Bhopal\\
		Bhopal-462066, India.}
	\email{saurabhk@iiserb.ac.in}
\thanks{}
\begin{abstract}
	 In this article, we establish weak type $(1,1)$ bounds for the variation and jump operators associated with the family of truncated singular integrals with rough kernels. This resolves an open question raised by Jones, Seeger and Wright (Trans. Amer. Math. Soc. (2008)). Moreover, as an immediate consequence of the variational estimate, we recover the weak type $(1,1)$ boundedness of the maximally truncated rough singular integral operator. 
\end{abstract}
	
\maketitle
	
\section{Introduction}\label{sec:intro}
A central research theme in harmonic analysis is to investigate whether a family of functions $\{F_t: t\in\mathcal I\}$, where $\mathcal I\subseteq \R$ is an interval, converges almost everywhere as the parameter $t$ tends to an endpoint of the interval $\mathcal I$. Generally, this involves a one-parameter family
of operators acting on a fixed function belonging to the space of Lebesgue integrable functions on $\R^d, d\geq 1.$ The classical approach to such investigations often factors through two ingredients: an a priori bound for the 
maximal function associated with the underlying family, and a dense subclass of the ambient function space on which the convergence is elementary. Variational and jump inequalities for the family of functions offer a stronger and more robust substitute for this scheme. The variational and jump inequalities control the  oscillations rather than controlling the size of the family. Consequently, these inequalities directly yield convergence at every point at which the relevant variational quantity is finite. 

For $1<q\leq\infty$, consider the $q-$variation associated with a family of functions $\{F_t:t\in\mathcal{I}\}$ defined as follows:
\[V_q(F_t(x):t\in\mathcal{I})=\begin{cases}\sup\limits_{N\in\N}\sup\limits_{\substack{t_1<t_2<\cdots t_N\\t_i\in\mathcal{I}}}\left(\sum_{i=1}^{N-1}|F_{t_{i+1}}(x)-F_{t_i}(x)|^q\right)^\frac{1}{q},&1<q<\infty\\
	\sup\limits_{t_1,t_2\in\mathcal{I}:t_1<t_2}|F_{t_1}(x)-F_{t_2}(x)|,&q=\infty.
\end{cases}\]
The $\lambda-$jump corresponding to the family of functions $\{F_t:t\in\mathcal{I}\}$ is defined as follows:
\begin{align*}
	N_\lambda(F_t(x):t\in\mathcal{I})=\sup\Big\{&N\in\N\cup\{0\}:\exists\;t_1,t_2,\cdots,t_N\in\mathcal{I},\;0<t_1<t_2<\cdots<t_N\\
	&\text{such that}\;|F_{t_{i+1}}(x)-F_{t_i}(x)|>\lambda\;\text{for}\;i=1,\cdots,N-1\Big\}.
\end{align*} 
We refer the reader to~\cite{JSW2008} for more details about the notion of variational and jump inequalities. For notational convenience, we write $V_q(F_t(x))=V_q(F_t(x):t\in(0,\infty))$ for the special case $\mathcal{I}=(0,\infty)$. Similarly, we set $N_\lambda(F_t(x))=N_\lambda(F_t(x):t\in(0,\infty))$. 

This article is concerned with the endpoint estimates for variational and jump operators for the family of truncations of singular integral operators with rough kernels. First, we review and discuss several key developments in the context of rough singular integral operators, their corresponding  maximal functions and the associated variational and jump inequalities. The study of variational inequalities originates in probability theory with the work of L\'epingle~\cite{Lep1976}. He proved a $q$-variational inequality for martingales in the range $q>2$, together with the fact that the range is optimal. Bourgain \cite{Bourgain1989} extended these ideas to ergodic theory, where he proved variational estimates in order to obtain pointwise ergodic theorems along arithmetic sets. Later, these ideas were developed further by Jones, Kaufman, Rosenblatt and Wierdl \cite{JKRW1998} and by Jones, Rosenblatt and Wierdl \cite{JRW2001,JRW2003} in the form of oscillation inequalities. The study of oscillation and variational inequalities for singular integrals was initiated by Campbell, Jones, Reinhold and Wierdl, first for the Hilbert transform \cite{CJRRW2000} and
then for singular integrals in higher dimensions \cite{CJRW2003}. Finally, Jones,
Seeger and Wright \cite{JSW2008} systematically extended the theory of variational and jump
inequalities in a more general setting. Since, the study of variational and jump inequalities in harmonic analysis and ergodic theory forms a vast subject, we do not intend to provide a comprehensive account of various developments in the subject; rather we will briefly discuss some of the key developments relevant to the theme of the paper in the following section. 

 \subsection{Rough singular integral operators} For $d\geq 2$, let $\Omega\in L^1(\Sp^{d-1})$ satisfy the mean zero condition 
 \begin{equation}\label{meanzero}
	\int\Omega(\theta)d\sigma(\theta)=0.
 \end{equation}
 The truncated singular integral operator with rough kernel $K(y)=\Omega(\frac{y}{|y|})|y|^{-d},~y\neq 0,$ is defined by 
\[T_{\Omega,\epsilon} f(x)=\int_{|x-y|>\epsilon}\Omega\left(\frac{x-y}{|x-y|}\right)|x-y|^{-d}f(y)dy.\]
Calder\'on and Zygmund \cite{CZ1956} considered the rough singular integral operator $T_\Omega$ and its maximal variant $T_\Omega^*$ defined by 
\[T_{\Omega} f(x)=\lim_{\epsilon\to0}T_{\Omega,\epsilon} f(x)\quad\text{and}\quad T_\Omega^*f(x)=\sup_{\epsilon>0}|T_{\Omega,\epsilon}f(x)|;\]
respectively. They proved the $L^p(\R^d)-$boundedness of both the operators for $1<p<\infty$ and $\Omega\in L\log L(\Sp^{d-1})$ via the method of rotations. We refer to the work of Duoandikoetxea and Rubio de Francia \cite{DR1986} for an alternate proof involving a double dyadic analysis and Fourier transform estimates. What makes the operators $T_{\Omega}$ and its variants delicate is that there is no regularity condition imposed on the function $\Omega$. In particular, it need not satisfy the H\"ormander condition, and the classical Calder\'on--Zygmund machinery, which
converts an $L^2-$boundedness into a weak type $(1,1)$ bound by exploiting the
regularity of the kernel away from the diagonal, is unavailable in this context. Therefore, establishing endpoint estimates requires more sophisticated techniques. Christ and Rubio de Francia \cite{CR1988} and independently Hofmann \cite{Hofmann1988} initially developed these ideas and showed that $T_\Omega$ maps $L^1(\R^d)$ to $L^{1,\infty}(\R^d)$ for $d=2$. The weak $(1,1)$ boundedness for $T_\Omega$ was extended for all dimensions $d\geq2$ by Seeger \cite{Seeger1996} using a microlocal analysis of the kernel. Such microlocal analysis of the kernel has proved valuable to obtain weak type endpoint results for a variety of operators.

The boundedness of the maximal operator $T_\Omega^*$ at the endpoint $p=1$ remained open until very recently. Honz\'ik \cite{Honzik2020} showed that $T_\Omega^*$ maps the space $L(\log\log L)^{2+\epsilon}(\R^d)$ to $L^{1,\infty}(\R^d)$ locally for $\epsilon>0$ and $\Omega\in L^\infty(\Sp^{d-1})$. A further improvement was given by the first author and Mohanty in \cite{BM2023}, where they showed that $T_\Omega^*$ maps the larger space $L\log\log L(\R^d)$ to $L^{1,\infty}(\R^d)$ for $\Omega\in L\log L(\Sp^{d-1})$. Recently, Lai \cite{Lai2025} showed that $T_\Omega^*$ indeed maps $L^1(\R^d)$ to $L^{1,\infty}(\R^d)$. 

In this paper, we are concerned with the variation and jump inequalities corresponding to the families of truncations of singular integrals with rough kernels. Campbell, Jones, Reinhold and Wierdl~\cite{CJRRW2000,CJRW2003} investigated the $L^p-$bounds of variational and jump inequalities for the family $\{T_{\Omega,\epsilon}:\;\epsilon\in (0,\infty)\}$. In particular, they proved the following bounds for $\Omega\in L\log L(\Sp^{d-1})$ and $1<p<\infty$:
\begin{align}
	\left\|V_q(T_{\Omega,\epsilon}f)\right\|_p\lesssim\|f\|_p,\quad\text{for}\;q>2,\label{strong:var}\\
	\left\|\lambda(N_\lambda(T_{\Omega,\epsilon}f))^\frac{1}{\rho}\right\|_p\lesssim\|f\|_p,\quad\text{for}\;\rho>2.\label{strong:jump}
\end{align}
Jones, Seeger and Wright \cite{JSW2008} established the strong type inequality \eqref{strong:jump} at the critical case $\rho=2$ for $\Omega\in L^r(\Sp^{d-1}),\;1<r\leq\infty$. The estimate \eqref{strong:jump} for $\rho=2$ and $\Omega\in L\log L(\Sp^{d-1})$ was obtained by Ding, Hong and Liu~\cite{DHL2017}. Thus, in the range $1<p<\infty$, the theory of variational and jump inequalities is complete. The situation at the endpoint $p=1$ is remarkably delicate, specifically in the context of rough kernels. Campbell, Jones, Reinhold and Wierdl~\cite{CJRRW2000,CJRW2003} proved the following weak type $(1,1)$ bounds, when $\Omega$ is a Lipschitz function:
\begin{align}
	\left|\left\{x\in\R^d:\;V_q(T_{\Omega,\epsilon}f(x))>\alpha\right\}\right|\lesssim\frac{\|f\|_1}{\alpha},\quad\text{for}\;q>2,\label{weak:var}\\
	\left|\left\{x\in\R^d:\;\lambda(N_\lambda(T_{\Omega,\epsilon}f(x)))^\frac{1}{\rho}>\alpha\right\}\right|\lesssim\frac{\|f\|_1}{\alpha},\quad\text{for}\;\rho>2.\label{weak:jump}
\end{align}
 Jones, Seeger and Wright~\cite{JSW2008} established the weak type $(1,1)$ inequality \eqref{weak:jump} at the critical case $\rho=2$ when $\Omega$ is a Lipschitz function. In the same paper, the question of weak type $(1,1)$ boundedness of the underlying operator in the absence of any smoothness assumption on $\Omega$ was posed and left open. In this paper, we resolve this question affirmatively. We would like to remark here that it is the endpoint case which makes this problem more delicate. The estimates for variational quantities at the endpoint require one to analyze the full family of truncations simultaneously, across all scales, rather than at a single optimising scale. 
 \section{Notation} The notation $A\lesssim B$ means that there exists a constant $C>0$ (independent of $A$ and $B$) such that $A\leq CB$.
\section{Main result of the paper}
In this paper, we resolve the question of endpoint bounds for variational and jump inequalities for the family of truncated rough singular integral operators $\{T_{\Omega,\epsilon}:\;\epsilon>0\}$ posed by Jones, Seeger and Wright~\cite{JSW2008}
affirmatively; this constitutes the main result of our work.
\begin{theorem}\label{mainthm}
	Let $\Omega\in L\log L(\Sp^{d-1})$ satisfying the mean zero condition \eqref{meanzero}. Then, for all $\alpha>0$, we have
	\begin{equation}\label{main:jump}
		\left|\left\{x\in\R^d:\lambda[N_\lambda(T_{\Omega,\epsilon} f(x))]^\frac{1}{2}>\alpha\right\}\right|\lesssim\frac{\|f\|_1}{\alpha}.
	\end{equation}
	Moreover for $q>2$, we have
	\begin{equation}\label{main:var}
		\left|\left\{x\in\R^d:\;V_q(T_{\Omega,\epsilon}f(x))>\alpha\right\}\right|\lesssim\frac{\|f\|_1}{\alpha}.
	\end{equation}
\end{theorem}
Since the variational quantity dominates the corresponding maximal function,
we recover the weak $(1,1)$ boundedness of $T_\Omega^*$ proved recently by Lai~\cite{Lai2025}
\begin{theorem}\label{cor:maximal}
	Let $\Omega\in L\log L(\mathbb S^{d-1})$ satisfying the mean zero condition \eqref{meanzero}. Then, for all $\alpha>0$, we have
	\begin{equation*}
		\left|\left\{x\in\R^d: T^*_{\Omega} f(x)>\alpha\right\}\right|\lesssim\frac{\|f\|_1}{\alpha}.
	\end{equation*}
\end{theorem}
	\section{Methodology of the proof of \Cref{mainthm}} 
	We first note that by \cite[Lemma 2.3]{MSZ2020}, the weak type $(1,1)$ estimate \eqref{main:var} for the $q-$variation operator will follow from the corresponding weak type estimate \eqref{main:jump} involving the jump function. Further, it is well known that in order to prove \eqref{main:jump}, it is enough to consider the short and long jumps operators separately via the inequality \eqref{shortlong}. The estimate for the short jump relies on a numerical inequality (\Cref{lemma:shortcounting}) that essentially reduces the short jumps at a fixed scale to a certain square function involving only dyadic terms. Such inequalities have been known to be successful in the context of $L^p-$estimates for $q-$variation operator associated with discrete Radon transforms for $1<p<\infty$ in \cite{MT2016,MST2017,MST2019,MSK2020}. Based on a similar philosophy, we establish a variant of this inequality in the form of \Cref{lemma:shortcounting}. To our knowledge, this is the first instance where such an inequality is also helpful to establish weak type bounds for the short jumps associated to singular integral operators at the endpoint $p=1$.

	For the long jump, the variational Rademacher-Menshov theorem is a key ingredient. This result was first obtained by Lewko and Lewko~\cite{LL2012}. We will use the following version of the variational Rademacher-Menshov theorem from Do, Oberlin and Palsson~\cite[Lemma 7.2]{DOP2017}.
	\begin{lemma}\label{lem:radamacher-menshov}
	Let $\{F_i\}_{i=1}^N$ be a sequence of functions on a measure space $X$ such that for all $\varepsilon_i\in\{-1,0,1\}$, we have
	\begin{equation*}
		\left\|\sum_{i=1}^N\varepsilon_iF_i\right\|_2\leq B.
	\end{equation*}
	Then, the following estimate holds 
	\begin{equation*}
		\left\|V_2\left(\sum_{i=1}^nF_i:1\leq n\leq N\right)\right\|_2\lesssim (\log N) B.
	\end{equation*}
\end{lemma}
Observe that Lemma~\ref{lem:radamacher-menshov} allows us to obtain $L^2-$bounds of the $2-$variation of partial sums of order $N$ from uniform $L^2-$bounds for arbitrary signed sums at the cost of a logarithmic factor $\log N$. This logarithmic loss is harmless as it is compensated by a geometric gain in the corresponding parameter. This geometric gain is based on the microlocal estimates of Seeger \cite{Seeger1996}. Our task is to combine the $L^1$ and $L^2$ estimates of Seeger \cite{Seeger1996} into a single $L^2$ estimate. This step relies on a careful interpolation argument, motivated by the approach of Honz\'ik~\cite{Honzik2020}. We provide the details of this estimate in \Cref{sec:microlocal}. This step yields uniform $L^2-$estimates for arbitrary signed sums with geometric gain mentioned above; consequently, the $L^2-$bound for the variational quantity are obtained using  Lemma~\ref{lem:radamacher-menshov}. 

The main distinction between the short and long jumps is the presence of multiscale analysis in the estimate for the long jump. The estimate for the short jumps relies on a single scale argument reminiscent of the ideas developed by Campbell, Jones, Reinhold and Wierdl~ \cite{CJRW2003}. In contrast, the estimates for the long jumps are based on an algorithm that controls the interaction of scales, pioneered in the work of Krause and Lacey \cite{KL2018,KL2020} for obtaining sparse bounds for maximal truncations of oscillatory Calder\'on-Zygmund operators. We also refer to \cite{IKM2025} for a simplified approach to proving endpoint estimates for such maximal oscillatory operators in dimension one. 

In the process of adapting the ideas of Krause and Lacey~\cite{KL2020}, we streamline certain aspects of their proof. First, we do away with the recursion argument employed there. Second, we introduce an effective decomposition of the bad functions based on the shifted dyadic grids. This allows us to use the $L^2-$bounds (\Cref{lemma:Seegercombined}) as a blackbox instead of arduously reproving similar estimates for a localised version of the truncations of rough singular integrals. As a consequence, this approach provides us with a new and simpler proof of the weak $(1,1)-$boundedness of $T_\Omega^*$. 

The proof of \Cref{mainthm} \eqref{main:jump} is given in \Cref{sec:proofmain}.

\section{Proof of \Cref{mainthm} \eqref{main:jump}}\label{sec:proofmain}
Let $f\in L^1(\R^d)$ and $\alpha>0$. By the Calder\'on-Zygmund decomposition, applied to $f$ at height $\alpha$, there exists a family $\mathcal{Q}$ of disjoint dyadic cubes such that
\begin{enumerate}
		\item $f=g+b$ with $b=\sum_{s\in\Z}B^{s}$, where $B^{s}=\sum_{Q\in\mathcal{Q}:|Q|=2^{sd}}b_{Q}$, 
		\item $\|g\|_\infty\lesssim\alpha$ and $\|g\|_{1}\leq\|f\|_1$,
		\item $\int b_{Q}(x)dx=0$ and $\supp{b_{Q}}\subseteq Q$,
		\item $\|b_{Q}\|_{1}\lesssim\alpha|Q|$ and $\sum_{Q\in\mathcal{Q}}|Q|\lesssim\frac{\|f\|_1}{\alpha}$. In particular, this implies $\sum_s\|B^{s}\|_{1}\lesssim\|f\|_1$.
\end{enumerate}
We define the exceptional set $E=\cup_{Q\in\mathcal{Q}}4Q$. Thus $|E|\lesssim\frac{\|f\|_1}{\alpha}$. By the $L^2-$boundedness \cite[Theorem 1.2 (ii)]{DHL2017} of the $\lambda-$jump operators associated to $\{T_{\Omega,\epsilon}:\epsilon>0\}$, we have 
\begin{align*}
	|\{x\in\R^d:\lambda[N_\lambda(T_{\Omega,\epsilon} g(x):\epsilon\in (0,\infty))]^\frac{1}{2}>\alpha\}|\lesssim\frac{1}{\alpha^2}\|g\|_2^2\lesssim\frac{\|f\|_1}{\alpha}.
\end{align*}
Therefore, it remains to estimate the bad part $b$. Let $\psi$ be a Schwartz function supported in the annulus $\{\frac{1}{4}\leq|y|\leq1\}$ and $\sum_{j\in\Z}\psi(2^{-j}y)=1$. We decompose the kernel as follows
\[K(y)=\Omega\left(\frac{y}{|y|}\right)|y|^{-d}=\sum_{j\in\Z}\Omega\left(\frac{y}{|y|}\right)|y|^{-d}\psi(2^{-j}y)=\sum_{j\in\Z}K_j(y).\]
We note that by the support of the kernels $K_j$, it follows that
\[\sum_{s<5}\sum_{j\in\Z}(K_j\chi_{|\cdot|>\epsilon})*B^{j-s}(x)=0,\quad\text{for}\; x\in E^c.\]
Thus, we have that
\[\lambda[N_\lambda(T_{\Omega,\epsilon}b(x):\epsilon\in (0,\infty))]^\frac{1}{2}\leq\lambda\left[N_{\lambda}\left(\sum_{s\geq5}\sum_{j\in\Z}(K_j\chi_{|\cdot|>\epsilon})*B^{j-s}(x):\epsilon\in (0,\infty)\right)\right]^\frac{1}{2}.\]
For each $s\geq5$, we write $K_j=S_j^{s}+H_j^{s}$, where
\begin{align*}
	S_j^{s}(y)&=K_j(y)\chi_{\left\{\left|\Omega\left(\frac{y}{|y|}\right)\right|>2^{\delta_0 s}\|\Omega\|_1\right\}}(y),\\
	H_j^{s}(y)&=K_j(y)\chi_{\left\{\left|\Omega\left(\frac{y}{|y|}\right)\right|\leq2^{\delta_0 s}\|\Omega\|_1\right\}}(y),
\end{align*}
for some fixed $\delta_0>0$. Thus, we obtain
\begin{align*}
	\lambda[N_\lambda(T_{\Omega,\epsilon}b:\epsilon\in (0,\infty))]^\frac{1}{2}(x)\leq&\lambda\left[N_\frac{\lambda}{2}\left(\sum_{s\geq5}\sum_{j\in\Z}(S_j^s\chi_{|\cdot|>\epsilon})*B^{j-s}(x):\epsilon\in (0,\infty)\right)\right]^\frac{1}{2}\\
	&+\lambda\left[N_\frac{\lambda}{2}\left(\sum_{s\geq5}\sum_{j\in\Z}(H_j^s\chi_{|\cdot|>\epsilon})*B^{j-s}(x):\epsilon>0\right)\right]^\frac{1}{2}.
\end{align*}
For the term involving the operators $S_j^s$ in the above, by the definition of $\lambda-$jumps and variation, we have that
\begin{align*}
	&\lambda\left[N_\frac{\lambda}{2}\left(\sum_{s\geq5}\sum_{j\in\Z}(S_j^s\chi_{|\cdot|>\epsilon})*B^{j-s}(x):\epsilon\in (0,\infty)\right)\right]^\frac{1}{2}\\
	\lesssim& V_1\left(\sum_{s\geq5}\sum_{j\in\Z}(S_j^s\chi_{|\cdot|>\epsilon})*B^{j-s}(x):\epsilon\in (0,\infty)\right)\\
	\leq&\sum_{s\geq5}\sum_{j\in\Z}|S_j^s|*|B^{j-s}|(x).
\end{align*}
Hence, we use Chebyshev's inequality to obtain
\begin{align*}
	&\left|\left\{x\in E^c:\lambda\left[N_\frac{\lambda}{2}\left(\sum_{s\geq5}\sum_{j\in\Z}(S_j^s\chi_{|\cdot|>\epsilon})*B^{j-s}(x):\epsilon\in (0,\infty)\right)\right]^\frac{1}{2}>\frac{\alpha}{2}\right\}\right|\\
	\lesssim&\frac{1}{\alpha}\sum_{s\geq5}\sum_{j\in\Z}\|S_j^s\|_1\|B^{j-s}\|_1\\
	\lesssim&\frac{1}{\alpha}\sum_{s\geq5}\sum_{j\in\Z}\|B^{j-s}\|_1\int_{|\Omega(\theta)|>2^{\delta_0 s}\|\Omega\|_1}\Omega(\theta)d\sigma(\theta)\\
	\lesssim&\frac{1}{\alpha}\|\Omega\|_{L\log L}\|f\|_1.
\end{align*}
For the term involving operators $H_j^s$, we dominate the $\lambda-$jumps by the sum of short and long $\lambda-$jumps. This is achieved via a standard argument; for a proof, see \cite[Lemma 1.3]{JSW2008}. As a consequence, we obtain the following pointwise inequality: 
\begin{align}\label{shortlong}
	&\lambda\left[N_\frac{\lambda}{2}\left(\sum_{s\geq5}\sum_{j\in\Z}(H_j^s\chi_{|\cdot|>\epsilon})*B^{j-s}(x):\epsilon\in (0,\infty)\right)\right]^\frac{1}{2}\nonumber\\
	\lesssim&~\mathcal{S}_2\left(\sum_{s\geq5}\sum_{j\in\Z}(H_j^s\chi_{|\cdot|>\epsilon})*B^{j-s}(x):\epsilon\in (0,\infty)\right)\nonumber\\
	&+\lambda\left[N_\frac{\lambda}{2}\left(\sum_{s\geq5}\sum_{j\in\Z}(H_j^s\chi_{|\cdot|>2^k})*B^{j-s}(x):k\in\Z\right)\right]^\frac{1}{2},
\end{align}
where $\mathcal{S}_2$ is the short jump operator defined as 
\[\mathcal{S}_2(F_\epsilon:\epsilon>0)(x)=\left(\sum\limits_{k\in\Z}|V_2(F_\epsilon:\epsilon\in[2^{k-1},2^{k}])(x)|^2\right)^\frac{1}{2}.\]
The second term in the expression \eqref{shortlong} is known as the long jump function. 
\subsection{Estimate for short jumps} 
We begin by establishing the numerical inequality (\Cref{lemma:shortcounting}) required to estimate the short jumps. Such an inequality essentially reduces the study of short jumps to estimating the differences between dyadic evaluations at the required scale. For this purpose, we introduce the following notation.
For $k,l,i\in\Z$, we set $u_{k,l,i}=2^{k-1}+2^{k-1-l} i$, $I_{k,l,i}=2^{k-1}+2^{k-1-l}(i+[0,1])$ and $U_L=\{u_{k,l,i},\;l=0,\cdots,L,\;i=0,\cdots,2^l-1\}$. We have the following.
\begin{lemma}\label{lemma:shortcounting}
	Given $k \in \mathbb{Z},\;L\in\N$ and a continuous function $\mathfrak{f}:\left[2^{k-1}, 2^{k}\right] \rightarrow \mathbb{C}$, we have
	\begin{align*}
		V_2\left(\mathfrak{f}(t): t \in\left[2^{k-1}, 2^{k}\right]\right)  \leq &\sqrt{2} \sum_{l = 0}^L\left(\sum_{i=0}^{2^l-1}\left|\mathfrak{f}(u_{k,l,i+1})-\mathfrak{f}(u_{k,l,i})\right|^2\right)^{\frac{1}{2}}\\
		&+\sqrt{2}\left(\sum_{i=0}^{2^L-1}\left(V_1(\mathfrak{f}(t):\;t\in I_{k,L,i})\right)^2\right)^\frac{1}{2}.
	\end{align*}
\end{lemma}
\begin{proof}
	Let $2^{k-1}\leq t_0\leq t_1\leq\cdots\leq t_N\leq 2^k$ be an arbitrary increasing sequence of real numbers. We define $\mathcal{N}=\{n\in\{0,\cdots,N-1\}:\;[t_n,t_{n+1}]\cap U_L\neq\emptyset\}$. We consider $a_n$ and $b_n$ be the smallest and largest dyadic numbers respectively lying in $[t_n,t_{n+1}]\cap U_L$ for $n\in\mathcal{N}$. By the triangle inequality, we have
	\begin{align*}
		\left(\sum_{n=0}^{N-1}\left|\mathfrak{f}(t_{n+1})-\mathfrak{f}(t_{n})\right|^2\right)^\frac{1}{2}\leq \Sigma_1+\Sigma_2,
	\end{align*}
	where
	\begin{align*}
		&\Sigma_1=\left(\sum_{n\in\mathcal{N}}\left|\mathfrak{f}(b_{n})-\mathfrak{f}(a_{n})\right|^2\right)^\frac{1}{2}\quad\text{and}\\
		&\Sigma_2=\left(\sum_{n\in\mathcal{N}}\left|\mathfrak{f}(t_{n+1})-\mathfrak{f}(b_{n})+\mathfrak{f}(a_{n})-\mathfrak{f}(t_{n})\right|^2+\sum_{n\notin\mathcal{N}}\left|\mathfrak{f}(t_{n+1})-\mathfrak{f}(t_{n})\right|^2\right)^\frac{1}{2}.
	\end{align*}
	\textbf{Estimate for $\Sigma_1$:} For each $n\in\mathcal{N}$, we choose a maximal collection $\mathcal{M}_n$ of dyadic intervals of lengths $\{2^{k-1-l},\;l=0,\cdots,L\}$ that cover $[a_n,b_n]$. We note that the collection $\mathcal{M}_n$ contains atmost two intervals of same length and no interval is contained in two collections $\mathcal{M}_n$. Thus, employing Minkowski and Cauchy-Schwarz inequalities, we obtain
	\begin{align*}
		\Sigma_1\leq&\sum_{l=0}^L\left(\sum_{n\in\mathcal{N}}\left|\sum_{I=[I^-,I^+]\in\mathcal{M}_n:|I|=2^{k-1-l}}[\mathfrak{f}(I^+)-\mathfrak{f}(I^-)]\right|^2\right)^{\frac{1}{2}}\\
		\leq&\sqrt{2}\sum_{l = 0}^L\left(\sum_{i=0}^{2^l-1}\left|\mathfrak{f}(u_{k,l,i})-\mathfrak{f}(u_{k,l,i+1})\right|^2\right)^{\frac{1}{2}}
	\end{align*}
	\textbf{Estimate for $\Sigma_2$:} We note that the intervals of the form $[t_n,a_n],\;[b_n,t_{n+1}]$ for $n\in\mathcal{N}$ and the intervals $[t_n,t_{n+1}]$ for $n\notin\mathcal{N}$ lie in exactly one of the intervals $\{I_{k,L,i},\;i=0,\cdots,2^L-1\}$. Hence, applying the Cauchy-Schwarz inequality, grouping the intervals appropriately and the embedding $\ell_1\subset\ell_2$ implies that
	\begin{align*}
		&\Sigma_2^2\\
		\leq&2\sum_{n\in\mathcal{N}}\left|\mathfrak{f}(t_{n+1})-\mathfrak{f}(b_{n})\right|^2+\sum_{n\in\mathcal{N}}\left|\mathfrak{f}(a_{n})-\mathfrak{f}(t_{n})\right|^2+\sum_{n\notin\mathcal{N}}\left|\mathfrak{f}(t_{n+1})-\mathfrak{f}(t_{n})\right|^2\\
		\leq&2\sum_{i=0}^{2^L-1}\left(\sum_{\substack{n\in\mathcal{N}:\\ [b_n,t_{n+1}]\subseteq I_{k,L,i}}}\hspace{-0.2cm}\left|\mathfrak{f}(t_{n+1})-\mathfrak{f}(b_{n})\right|^2+\sum_{\substack{n\in\mathcal{N}:\\ [t_n,a_n]\subseteq I_{k,L,i}}}\hspace{-0.2cm}\left|\mathfrak{f}(a_{n})-\mathfrak{f}(t_{n})\right|^2\right.\\
		&\hspace{1.2cm}\left.+\sum_{\substack{n\notin\mathcal{N}:\\ [t_n,t_{n+1}]\subseteq I_{k,L,i}}}\hspace{-0.2cm}\left|\mathfrak{f}(t_{n+1})-\mathfrak{f}(t_{n})\right|^2\right)\\
		\leq&2\sum_{i=0}^{2^L-1}\left(V_1(\mathfrak{f}(t):\;t\in I_{k,L,i})\right)^2.
	\end{align*}
	This concludes the proof of \Cref{lemma:shortcounting}.
\end{proof}
We now proceed with the proof of the weak type estimate involving short jumps. Since the support of $H_j^s$ is contained in the annulus $\{2^{j-2}\leq|y|\leq2^j\}$, the estimate of $\mathcal{S}_2(T_{\Omega,\epsilon}b:\epsilon\in (0,\infty))$ is reduced to proving
\begin{equation*}
	\left|\left\{x\in E^c:\sum_{s\geq5}\left(\sum\limits_{k\in\Z}|V_2((H^s_{k+j}\chi_{|\cdot|>\epsilon})*B^{k+j-s}(x):\epsilon\in[2^{k-1},2^{k}])|^2\right)^\frac{1}{2}>\alpha\right\}\right|\lesssim\frac{\|f\|_1}{\alpha},
\end{equation*}
for $j=0,1$. We will prove the inequality for $j=0$, the other estimate is similar. By Chebyshev's and triangle inequality, it is enough to prove the following for $k\in\Z$ and $s\geq5$,
\begin{equation*}
	\left\|V_2((H_k^s\chi_{|\cdot|>\epsilon})*B^{k-s}:\epsilon\in[2^{k-1},2^{k}])\right\|_2^2\lesssim2^{-\delta' s}\alpha\sum_{Q\in\mathcal{Q}:|Q|=2^{(k-s)d}}\|b_Q\|_1,
\end{equation*}
for some $\delta'>0$. We set $E_{k,l,i}=\{y\in\R^d:\;|y|\in I_{k,l,i}\}$. We use \Cref{lemma:shortcounting} with $L=\lfloor\eta s\rfloor$ (for some fixed $0<\eta<1$ to be determined later) to obtain
\begin{align*}
	&\left\|V_2((H_k^s\chi_{|\cdot|>\epsilon})*B^{k-s}:\epsilon\in[2^{k-1},2^{k}])\right\|_2^2\lesssim\Sigma_1^k+\Sigma_2^k,
\end{align*}
where
\begin{align*}
	\Sigma_1^k= & \left\|\sum_{l = 0}^{\lfloor\eta s\rfloor}\left(\sum_{i=0}^{2^l-1}\left|(H_k^s\chi_{E_{k,l,i}})*B^{k-s}\right|^2\right)^{\frac{1}{2}}\right\|_2^2\quad\text{and}\\
	\Sigma_2^k=& \sum_{i=0}^{2^{\lfloor\eta s\rfloor}-1}\left\|V_1((H_k^s\chi_{|\cdot|>\epsilon})*B^{k-s}:\;\epsilon\in I_{k,\lfloor\eta s\rfloor,i})\right\|_2^2
\end{align*}
\subsubsection{\textbf{Estimate for $\Sigma_1^k$}:}\label{subsec:sigma1} To apply \Cref{lemma:Seegercombined}, we will need a radial regularization of our kernel similar to the approach in \cite[pages 102-103]{Seeger1996}. Let $\phi\in C_c^\infty(\R)$ with $\int\phi=1$ and $\supp{\phi}\subset[-2^{-10},2^{-10}]$. We set $\phi_k(r)=2^{-k}\phi(2^{-k}r)$ and define $\mathds{1}_{E_{k,l,i}}(y)=\chi_{I_{k,l,i}}*\phi_{k-\lfloor\eta s\rfloor}(|y|),\;y\in\R^d$. We note that
\[\sup_{0\leq l\leq 10d}2^{k(d+l)}|\frac{\partial^l}{\partial r^l}(H_k^s\mathds{1}_{E_{k,l,i}})(r\theta)|\lesssim\|\Omega\|_12^{(\delta_0+10d\eta) s}.\]
Hence, an application of \Cref{lemma:Seegercombined} yields
\begin{align}\label{est:short1}
	\sum_{l = 0}^{\lfloor\eta s\rfloor}\sum_{i=0}^{2^l-1}\left\|(H_k^s\mathds{1}_{E_{k,l,i}})*B^{k-s}\right\|_2^2\lesssim&\alpha\|\Omega\|_1^22^{-(\delta-21d\eta-2\delta_0)s}s\sum_{Q\in\mathcal{Q}:|Q|=2^{(k-s)d}}\|b_Q\|_1.
\end{align}
Next, in order to bound the error term, we record the following straightforward observation concerning the difference between the smoothed and the unsmoothed truncation.
\begin{equation}\label{est:error}
	|\chi_{E_{k,l,i}}(y)-\mathds{1}_{E_{k,l,i}}(y)|\lesssim\chi_{\partial E_{k,l,i}}(y),\;y\in\R^d,
\end{equation}
where $\partial E_{k,l,i}$ is the thickening of the boundary of the annulus $E_{k,l,i}$ given by 
\[\partial E_{k,l,i}=\left\{y\in\R^d:|y|\in \bigcup_{t\in\{0,1\}}u_{k,l,i+t}+[-2^{k-5-\lfloor\eta s\rfloor},2^{k-5-\lfloor\eta s\rfloor}]\right\}.\]
\sloppy The inequality \eqref{est:error} along with the bound $|H^s_{k}(y)|\lesssim\|\Omega\|_12^{\delta_0 s}2^{-kd}$ and the disjointness of the cubes in the collection $\mathcal{Q}$ implies that
\begin{align*}
	\left|(H_k^s(\chi_{E_{k,l,i}}-\mathds{1}_{E_{k,l,i}}))*B^{k-s}(x)\right|\lesssim&\|\Omega\|_12^{\delta_0 s}2^{-kd}(\chi_{\partial E_{k,l,i}})*|B^{k-s}|(x)\\
	\lesssim&\|\Omega\|_12^{\delta_0 s}2^{-kd}\sum_{\substack{Q\in\mathcal{Q}:\\|Q|=2^{(k-s)d}\\(x-\partial E_{k,l,i})\cap Q\neq\emptyset}}\|b_Q\|_1\\
	\lesssim&\alpha\|\Omega\|_12^{\delta_0 s}2^{-kd}\sum_{\substack{Q\in\mathcal{Q}:\\|Q|=2^{(k-s)d}\\(x-\partial E_{k,l,i})\cap Q\neq\emptyset}}|Q| \\ 
	\lesssim&\alpha\|\Omega\|_12^{\delta_0 s}2^{-kd}|\partial E_{k,l,i}|\\
	\lesssim&\alpha\|\Omega\|_12^{-(\eta-\delta_0) s}.
\end{align*}
Therefore, using the bounded overlapping property of the annuli $\{3E_{k,l,i}:\;i=0,\cdots,2^l-1\}$, it follows that
\begin{align}\label{est:short2}
	&\sum_{l = 0}^{\lfloor\eta s\rfloor}\sum_{i=0}^{2^l-1}\left\|(H_k^s(\chi_{E_{k,l,i}}-\mathds{1}_{E_{k,l,i}}))*B^{k-s}\right\|_2^2\nonumber\\
	\lesssim&\alpha\|\Omega\|_12^{-(\eta-\delta_0) s}\sum_{l = 0}^{\lfloor\eta s\rfloor}\sum_{i=0}^{2^l-1}\left\|(H_k^s\chi_{3E_{k,l,i}})*B^{k-s}\right\|_1\nonumber\\
	\lesssim&\alpha\|\Omega\|_12^{-(\eta-\delta_0) s}\sum_{l = 0}^{\lfloor\eta s\rfloor}\left\|H_k^s\right\|_1\|B^{k-s}\|_1\nonumber\\
	\lesssim&\alpha\|\Omega\|_1^2 s2^{-(\eta-\delta_0) s}\|B^{k-s}\|_1.
\end{align}
Thus, by an application of Cauchy-Schwarz inequality along with inequalities \eqref{est:short1} and \eqref{est:short2}, we have
\begin{align*}
	\Sigma_1^k &\lesssim s \sum_{l = 0}^{\lfloor\eta s\rfloor}\sum_{i=0}^{2^l-1}\left\|(H_k^s\chi_{E_{k,l,i}})*B^{k-s}\right\|_2^2\\
	&\lesssim s \sum_{l = 0}^{\lfloor\eta s\rfloor}\sum_{i=0}^{2^l-1}\left\|(H_k^s\mathds{1}_{E_{k,l,i}})*B^{k-s}\right\|_2^2+s \sum_{l = 0}^{\lfloor\eta s\rfloor}\sum_{i=0}^{2^l-1}\left\|(H_k^s(\chi_{E_{k,l,i}}-\mathds{1}_{E_{k,l,i}}))*B^{k-s}\right\|_2^2\\
	&\lesssim\alpha\|\Omega\|_1^2s^2(2^{-(\delta-21d\eta-2\delta_0)s}+2^{-(\eta-\delta_0) s})\sum_{Q\in\mathcal{Q}:|Q|=2^{(k-s)d}}\|b_Q\|_1.
\end{align*}
Choosing $\eta,\delta_0$ such that $21d\eta+2\delta_0<\frac{\delta}{10}$ and $0<\eta-\delta_0<\frac{\delta}{100}$ produces the required decay.

\subsubsection{\textbf{Estimate for $\Sigma_2^k$}:} The quantity involving variation is estimated by using the nesting structure of the truncation operators $\{(H_k^s\chi_{|\cdot|>\epsilon}),\;\epsilon>0\}$. An argument similar to \Cref{subsec:sigma1} then implies that
\begin{align*}
	&V_1((H_k^s\chi_{|\cdot|>\epsilon})*B^{k-s}(x):\;\epsilon\in[2^{k-1}+2^{k-1-\lfloor\eta s\rfloor}i,2^{k-1}+2^{k-1-\lfloor\eta s\rfloor}(i+1)])\\
	\leq&|H^s_{k}\chi_{E_{k,\lfloor\eta s\rfloor,i}}|*|B^{k-s}|(x)\\
	\lesssim&\alpha\|\Omega\|_12^{\delta_0 s}2^{-kd}\sum_{\substack{Q\in\mathcal{Q}:\\|Q|=2^{(k-s)d}\\(x-E_{k,\lfloor\eta s\rfloor,i})\cap Q\neq\emptyset}}|Q| \\ 
	\lesssim&\alpha\|\Omega\|_12^{\delta_0 s}2^{-kd}|E_{k,\lfloor\eta s\rfloor,i}|\\
	\lesssim&\alpha\|\Omega\|_12^{-(\eta-\delta_0) s}.
\end{align*}
Therefore, it follows that
\begin{align*}
	\Sigma_2^k&\lesssim\sum_{i=0}^{2^{\lfloor\eta s\rfloor}-1}\||H^s_{k}\chi_{E_{k,\lfloor\eta s\rfloor,i}}|*|B^{k-s}|(x)\|_2^2\\
	&\lesssim\alpha\|\Omega\|_12^{-(\eta-\delta_0) s}\sum_{i=0}^{2^{\lfloor\eta s\rfloor}-1}\|H^s_{k}\chi_{E_{k,\lfloor\eta s\rfloor,i}}\|\|B^{k-s}\|_1\\
	&\lesssim\alpha\|\Omega\|_1^22^{-(\eta-\delta_0) s}\sum_{Q\in\mathcal{Q}:|Q|=2^{(k-s)d}}\|b_Q\|_1.
\end{align*}
This concludes the estimate for the short jumps.

\subsection{Estimate for long jumps}In this section we prove the weak $(1,1)$ bound for the long jumps. We note that 
\begin{align*}
	&\lambda\left[N_\frac{\lambda}{2}\left(\sum_{s\geq5}\sum_{j\in\Z}(H_j^s\chi_{|\cdot|>2^k})*B^{j-s}:k\in\Z\right)\right]^\frac{1}{2}(x)\\
	\lesssim&\lambda\Big[N_{\frac{\lambda}{2}}\Big(\sum_{s\geq5}\sum_{j\geq k}H_j^s*B^{j-s}:k\in\Z\Big)(x)\Big]^\frac{1}{2}\\
	&+\lambda\Big[N_{\frac{\lambda}{2}}\Big(\sum_{s\geq5}(H_k^s\chi_{2^{k-1}\leq|.|\leq2^k})*B^{k-s}:k\in\Z\Big)(x)\Big]^\frac{1}{2}\\
	=:&L_1+L_2
\end{align*}
\subsubsection{\textbf{Estimate for $L_2$}:} The estimate for $L_2$ is comparatively easier as there is no interaction between scales $k$. Indeed, by Chebyshev's and triangle inequality, we have
\begin{align*}
	&\left|\left\{\lambda\Big[N_{\frac{\lambda}{2}}\Big(\sum_{s\geq5}(H_k^s\chi_{2^{k-1}\leq|.|\leq2^k})*B^{k-s}:k\in\Z\Big)(x)\Big]^\frac{1}{2}>\frac{\alpha}{4}\right\}\right|\\
	\lesssim&\frac{1}{\alpha^2}\left(\sum_{s\geq5}\left(\sum_{k\in\Z}\|(H_k^s\chi_{2^{k-1}\leq|.|\leq2^k})*B^{k-s}\|_2^2\right)^\frac{1}{2}\right)^2\\
	\lesssim&\frac{\|\Omega\|_1}{\alpha}\left(\sum_{s\geq5}2^{-(\delta-\delta_0)s/2}\left(\sum_{k\in\Z}\|B^{k-s}\|_1\right)^\frac{1}{2}\right)^2\\
	\lesssim&\frac{1}{\alpha}\|\Omega\|_1\|f\|_1,
\end{align*}
where we have used \Cref{lemma:Seegercombined} in the second last step.
\subsubsection{\textbf{Estimate for $L_1$}:} For the estimate of $L_1$, we note that
\begin{align*}
	&\left|\left\{\lambda\Big[N_{\frac{\lambda}{2}}\Big(\sum_{s\geq5}\sum_{j\geq k}H_j^s*B^{j-s}:k\in\Z\Big)(x)\Big]^\frac{1}{2}>\frac{\alpha}{4}\right\}\right|\\
	&\leq\left|\left\{\sum_{s\geq5}V_2\left(\sum_{j\geq k}H_j^s*B^{j-s}:k\in\Z\right)(x)>\frac{\alpha}{4}\right\}\right|\\
	&\leq\sum_{s\geq5}\left|\left\{V_2\left(\sum_{j\geq k}H_j^s*B^{j-s}:k\in\Z\right)(x)>{c_12^{-\delta_1s}\alpha}\right\}\right|,
\end{align*}
where $c_1,\delta_1$ are fixed constants such that $c_1\sum_{s\geq5}2^{-\delta_1 s}=\frac{1}{4}$. 

We will work with a fixed $s\geq5$. We set $K_{j,m}=2^j([0,1)^d+m)$. As earlier, we write 
\[H_j^s*B^{j-s}=\sum_{m\in\Z^d}H_j^s*B^{j-s,m},\quad \text{where}\; B^{j-s,m}=\sum_{|Q|=2^{(j-s)d}:Q\subseteq K_{j,m}}b_Q.\]
We note that each function $H_j^s*B^{j-s,m}$ is supported in the cube $2K_{j,m}$. It is easy to see that for each fixed $j\in\Z$, the cubes $\{2K_{j,m}\}_{m\in\Z^d}$ have bounded overlap but they lack the nesting structure. We will further subdivide the cubes $\{K_{j,m},j\in\Z,m\in\Z^d\}$ into finitely many subcollections based on the shifted dyadic grids defined as follows. For $u\in\{0,\frac{1}{3},\frac{2}{3}\}^d$, we define the dyadic grids $\mathcal{D}_u$ as  
\[\mathcal{D}_u=\{2^j([0,1]^d+m+u),\;j\in\Z,\;m\in\Z^d\}.\]
We will require the following fact from \cite[Lemma 2.5]{HLP2013}.
\begin{lemma}[\cite{HLP2013}]
	For any cube $K$, there exist $u \in\left\{0, \frac{1}{3}, \frac{2}{3}\right\}^d$ and $R \in \mathcal{D}_u$ such that $K \subset R,\;2K \subset \widetilde{R}$, and $3 \ell(K)<\ell(R) \leq6\ell(K)$, where $\widetilde{R}$ denote the unique cube in $\mathcal{D}_u$ containing $R$ with twice the length of $R$.
\end{lemma}
By the above lemma, for each $K_{j,m}$, there exist a $u\in\left\{0, \frac{1}{3}, \frac{2}{3}\right\}^d$ and $\widetilde{K}_{j,m}\in\mathcal{D}_u$ such that $2K_{j,m}\subset\widetilde{K}_{j,m}$ and $|\widetilde{K}_{j,m}|\lesssim|{K}_{j,m}|$. Based on this observation, we have the following decomposition:
\[B^{j-s,m}=\sum_{u\in\{0,\frac{1}{3},\frac{2}{3}\}^d}B^{j-s,m}_u,\;\text{where}\]
\[B^{j-s,m}_u=\begin{cases} 
      			B^{j-s,m}, & \;\text{if}\;\widetilde{K}_{j,m}\in\mathcal{D}_u \\
      			0, & \;\text{otherwise}.
   				\end{cases}
\]
Thus, it is enough to estimate
\[\left|\left\{V_2\left(\sum_{j\geq k}\sum_{m\in\Z^d}H_j^s*B^{j-s,m}_u(x):k\in\Z\right)>{c_1 3^{-d} 2^{-\delta_1s}\alpha}\right\}\right|,\]
for a fixed $u\in\left\{0,\frac{1}{3},\frac{2}{3}\right\}^d$. The advantage of the above expression is that $\supp{H_j^s*B^{j-s}_u}\subset\widetilde{K}_{j,m}$, where $\widetilde{K}_{j,m}\in\mathcal{D}_u$ and any such two $\widetilde{K}_{j,m}$ and $\widetilde{K}_{j',m'}$ are either disjoint or one is contained in another. We now discard the points where the interaction across the scales $j$ is large. This leaves us with points having a controlled overlap in the scales with respect to the parameter $s$. To do so, we initially set 
\begin{align*}
	\mathcal{C}_0^u&=\{(j,m):\exists Q\in\mathcal{Q},|Q|=2^{(j-s)d},\;Q\subseteq K_{j,m}\;\text{and}\;\widetilde{K}_{j,m}\in\mathcal{D}_u\}\;\text{and}\\
	\mathcal{X}^u&=\bigcup\limits_{(j,m)\in\mathcal{C}^u_0}\widetilde{K}_{j,m}.
\end{align*}
It is easy to see that $\supp{\sum_{j\in\Z}\sum_{m\in\Z^d}H_j^s*B^{j-s,m}_u}\subseteq \mathcal{X}^u$. We decompose the set $\mathcal{X}^u=\mathcal{X}^u_1\cup\mathcal{X}^u_2$, where
\[\mathcal{X}^u_1=\left\{x\in\mathcal{X}^u:\sum_{(j,m)\in\mathcal{C}^u_0}\chi_{\widetilde{K}_{j,m}}(x)>2^{s(d+\epsilon_0)}\right\}\quad\text{and}\quad\mathcal{X}^u_2=\mathcal{X}^u\setminus\mathcal{X}^u_1,\]
and $\epsilon_0>0$ is a small fixed constant. Here, we would like to point out that the choice of the set $\mathcal{X}^u_1$ based on $\epsilon_0>0$ allows us to do away with the recursion argument employed in \cite{KL2020}. We write
\begin{align*}
	&\left|\left\{V_2\left(\sum_{j\geq k}\sum_{m\in\Z^d}H_j^s*B^{j-s,m}_u(x):k\in\Z\right)>c_1 3^{-d} 2^{-\delta_1s}\alpha\right\}\right|\\
	\lesssim&|\mathcal{X}^u_1|+\left|\left\{x\in\mathcal{X}^u_2:V_2\left(\sum_{j\geq k}\sum_{m\in\Z^d}H_j^s*B^{j-s,m}_u(x):k\in\Z\right)>c_1 3^{-d} 2^{-\delta_1s}\alpha\right\}\right|.
\end{align*}
We note that for any $(j,m)\in\mathcal{C}^u_0$, there exists atleast one cube $Q_{j,m}\in\mathcal{Q}$ with $|Q_{j,m}|=2^{(j-s)d}$ and $Q_{j,m}\subset K_{j,m}$. Thus, by Chebyshev's inequality, we have
\begin{align*}
	|\mathcal{X}^u_1|&\leq\frac{1}{2^{s(d+\epsilon_0)}}\sum_{(j,m)\in\mathcal{C}^u_0}|\widetilde{K}_{j,m}|\\
	&\lesssim\frac{1}{2^{s\epsilon_0}}\sum_{(j,m)\in\mathcal{C}^u_0}|Q_{j,m}|\\
	&\lesssim\frac{2^{-s\epsilon_0}}{\alpha}\sum\|b_Q\|_1.
\end{align*}
Let $\mathcal{C}^u=\{(j,m)\in\mathcal{C}_0^u:\widetilde{K}_{j,m}\nsubseteq\mathcal{X}^u_1\}$. Then, it is easy to see that for $x\in\mathcal{X}^u_2$, we have
\[\sum_{j\geq k}\sum_{m\in\Z}H_j^s*B^{j-s,m}_u(x)=\sum_{(j,m)\in\mathcal{C}^u,\;j\geq k}H_j^s*B^{j-s,m}_u(x).\]
The advantage of working with the collection $\mathcal{C}^u$ is that for any fixed $(j,m)\in\mathcal{C}^u$, there are atmost $2^{s(d+\epsilon_0)}$ indices $(j',m')\in\mathcal{C}^u$ such that $\widetilde{K}_{j,m}\subset \widetilde{K}_{j',m'}$. Based on this observation, we subdivide the collection $\mathcal{C}^u$ as follows. For any $(j,m)\in\mathcal{C}^u$, we say $(j,m)\in \mathcal{C}^u_1$ if there is no other index $(j',m')\in\mathcal{C}^u$ such that $\widetilde{K}_{j',m'}\subseteq \widetilde{K}_{j,m}$. Inductively, we define $\mathcal{C}^u_l$ to be the collection of all indices $(j,m)\in\mathcal{C}^u\setminus\cup_{i=1}^{l-1}\mathcal{C}^u_i$ such that there does not exist any index $(j',m')\in\mathcal{C}^u\setminus\cup_{i=1}^{l-1}\mathcal{C}^u_i$ such that $\widetilde{K}_{j',m'}\subset \widetilde{K}_{j,m}$. By the observation before, the process terminates at the step $l=2^{s(d+\epsilon_0)}$ and we have $\mathcal{C}^u=\bigcup\limits_{l=1}^{2^{s(d+\epsilon_0)}}\mathcal{C}^u_l$. Thus, we have that
\begin{align*}
	&\left|\left\{x\in\mathcal{X}^u_2:V_2\left(\sum_{j\geq k}\sum_{m\in\Z^d}H_j^s*B^{j-s,m}_u(x):k\in\Z\right)>c_1 3^{-d} 2^{-\delta_1s}\alpha\right\}\right|\\
	&\lesssim\frac{2^{2\delta_1s}}{\alpha^2}\left\|V_2\left(\sum_{j\geq k:\;(j,m)\in\mathcal{C}^u}H_j^s*B^{j-s,m}_u:k\in\Z\right)\right\|_2^2\\
	&=\frac{2^{2\delta_1s}}{\alpha^2}\left\|V_2\left(\sum_{i=1}^{l}\left(\sum_{(j,m)\in\mathcal{C}^u_i}H_j^s*B^{j-s,m}_u\right):1\leq l\leq2^{s(d+\epsilon_0)}\right)\right\|_2^2\\
	&\lesssim\frac{s^22^{2\delta_1s}}{\alpha^2}\sup_{\varepsilon_l\in\{-1,0,1\}}\left\|\sum_{l=1}^{2^{s(d+\epsilon_0)}}\varepsilon_l\left(\sum_{(j,m)\in\mathcal{C}^u_l}H_j^s*B^{j-s,m}_u\right)\right\|_2^2\\
	&=\frac{s^22^{2\delta_1s}}{\alpha^2}\sup_{\varepsilon_l\in\{-1,0,1\}}\left\|\sum_{j\in\Z}H_j^s*\left(\sum_{l=1}^{2^{s(d+\epsilon_0)}}\sum_{m:(j,m)\in\mathcal{C}^u_l}\varepsilon_lB^{j-s,m}_u\right)\right\|_2^2\\
	&\lesssim\frac{s^2 2^{-s(\delta-2\delta_0-2\delta_1)}}{\alpha}\left\|\Omega\right\|_1\sum_{Q\in\mathcal{Q}}\|b_Q\|_1,
\end{align*}
where we used \Cref{lemma:Seegercombined} in the final step.
Choosing $\delta_0,\delta_1\leq \frac{\delta}{100}$ and summing in $u$ and $s$ concludes the proof.
\appendix
\renewcommand{\thesection}{\Alph{section}.}
\section{Microlocal estimates for rough singular integrals}\label{sec:microlocal}

\subsection*{Estimates of Seeger:} Let $H_0$ be a function such that
\begin{itemize}
	\item $\supp{H_0}\subseteq\{2^{-1}\leq|x|\leq2\}$.
	\item We have the following control over the radial derivative
	\[\sup_{0\leq l\leq 10d}r^{d+l}|\frac{\partial^l}{\partial r^l}H_0(r\theta)|\leq\mathscr{M},\]
	uniformly in $r$ and $\theta$. 
\end{itemize}
We set $H_j(x)=2^{-jd}H_0(2^{-j}x)$. For $s\geq5$, let $\mathscr{E}^s=\{e_v^s\}$ be a collection of unit vectors such that 
\begin{itemize}
	\item $|e_{v_1}^s-e_{v_2}^s|>2^{-s-10}d^{-1}$ for all $e_{v_1}^s,e_{v_2}^s\in\mathscr{E}^s$.
	\item For each $\theta\in\Sp^{d-1}$, there exists $e_v^s\in\mathscr{E}^s$ such that $|\theta-e_v^s|\leq2^{-s-1}$.
\end{itemize}
It is easy to see that there exists a collection of disjoint measurable sets $E_v^s\subset\Sp^{d-1}$ such that $e_v^s\in E_v^s$, $\sigma(E_v^s)\leq2^{-s(d-1)}$ and $\cup_v E_v^s=\Sp^{d-1}$. Let $H_{j,v}^s(x)=H_j(x)\chi_{E_v^s}(x/|x|)$. For a fixed parameter $\kappa$ with $0<\kappa<1$, we have the following decomposition from \cite[(2.2)]{Seeger1996},
	\[H_j=\Gamma_j^s+(H_j-\Gamma_j^s),\]
	where $\widehat{\Gamma_j^s}(\xi)=\sum_v \psi(2^{s(1-\kappa)}\langle \xi,e_v^s\rangle/|\xi|)\widehat{H_{j,v}^s}(\xi)$ and $\psi\in C_c^\infty(\R)$ is supported in $[-4,4]$ and $\psi(t)=1$ for $t\in[-2,2]$.
	
	We have the following estimates, which are essentially an amalgamation of Lemmas 2.1 and 2.2 in \cite{Seeger1996}.
	\begin{lemma}[\cite{Seeger1996}]\label{lemma:Seeger}
		Let $\alpha>0$, $\mathcal{I}\subset\Z$, and $\mathcal{Q}$ be a collection of cubes with disjoint interior such that for each $Q\in\mathcal{Q}$ there exists a function $b_Q\in L^1(Q)$ such that
		\[\int|b_Q(x)|dx\leq\alpha|Q|.\]
		Let $B^k=\sum_{Q\in\mathcal{Q}:|Q|=2^{dk}}b_Q$. Then, for $s\geq5$, we have
		\[\left\|\sum_{j\in\mathcal{I}}\Gamma_j^s*B^{j-s}\right\|_2^2\lesssim\mathscr{M}^22^{-s(1-\kappa)}\alpha\sum_{Q\in\mathcal{Q}}\|b_Q\|_1.\]
		Moreover, if $\int b_Q=0$ for each $Q\in\mathcal{Q}$. Then, for $s\geq5$ and $0\leq\kappa'\leq1$, we have
		\[\left\|\sum_{j\in\mathcal{I}}(H_j-\Gamma_j^s)*B^{j-s}\right\|_1\lesssim\mathscr{M}(2^{-s\kappa'}+2^{s(d+(\kappa'-\kappa)10d)})\sum_{Q\in\mathcal{Q}}\|b_Q\|_1.\]
	\end{lemma}
	We combine the $L^1$ and $L^2$ estimates in \Cref{lemma:Seeger} into a single $L^2$ estimate in the following manner. 
	\begin{lemma}\label{lemma:Seegercombined}
		Let $\alpha>0$, $\mathcal{I}\subset\Z$, and $\mathcal{Q}$ be a collection of cubes with disjoint interior such that for each $Q\in\mathcal{Q}$ there exists a function $b_Q\in L^1(Q)$ such that
		\[\int|b_Q(x)|dx\leq\alpha|Q|.\]
		Let $B_k=\sum_{Q\in\mathcal{Q}:|Q|=2^{dk}}b_Q$. Then, there exists $\delta>0$ such that for $s\geq5$, we have
		\[\left\|\sum_{j\in\mathcal{I}}H_j*B^{j-s}\right\|_2^2\lesssim\mathscr{M}^22^{-s\delta}\alpha\sum_{Q\in\mathcal{Q}}\|b_Q\|_1.\]
	\end{lemma}
	\begin{proof}
		We write 
		\begin{align*}
			\sum_{j\in\mathcal{I}}H_j*B^{j-s}=\sum_{j\in\mathcal{I}}\Gamma_j^s*B^{j-s}+\sum_{j\in\mathcal{I}}(H_j-\Gamma_j^s)*B^{j-s}=:A+B.
		\end{align*}
		We observe that
		\begin{align*}
			&\int|(A+B)(x)|^2dx\\
			\leq&\left(\int_{|B(\cdot)|\leq2^{\rho s}\alpha}+\int_{|A(\cdot)|>2^{\rho s-1}\alpha}+\int_{|(A+B)(\cdot)|>2^{\rho s-1}\alpha}\right)|(A+B)(x)|^2dx\\
			=:&I+II+III.
		\end{align*}
		\textit{Estimate for I:} By \Cref{lemma:Seeger} with $\kappa=\frac{1}{2},\;\kappa'=\frac{1}{4},\;N=4d+1$, we have
		\begin{align*}
			I&\lesssim \int|A(x)|^2dx+\int_{|B(\cdot)|\leq2^{\rho s}\alpha}|B(x)|^2dx\\
			&\lesssim\int|A(x)|^2dx+2^{\rho s}\alpha\int_{|B(\cdot)|\leq2^{\rho s}\alpha}|B(x)|dx\\
			&\lesssim\mathscr{M}^22^{-\frac{s}{8}}\alpha\sum\|b_Q\|_1,
		\end{align*}
		where we used $\rho$ such that $\rho<\frac{1}{8}$. To estimate $II$ and $III$, we will need the following claim:
		\begin{equation}\label{claim:sublevel}
			\left|\left\{x\in\R^d:|\sum_{j\in\mathcal{I}}H_j*B^{j-s}|>\lambda\alpha\right\}\right|\lesssim \frac{\mathscr{M}}{\lambda^3\alpha}\sum\|b_Q\|_1,\quad\text{for}\;\lambda>0.
		\end{equation}
		Assuming the above estimate, we estimate $II$ and $III$.

		\textit{Estimate for $II$:} By Chebyshev's inequality, \Cref{lemma:Seeger} ($\kappa=\frac{1}{2}$) and \eqref{claim:sublevel}, we have
		\begin{align*}
			II&\lesssim\int_{0}^\infty\min\left\{\left|\left\{|A(x)|>2^{\rho s-1}\alpha\right\}\right|,\left|\left\{|(A+B)(x)|>\lambda\right\}\right|\right\}\lambda d\lambda\\
			&\lesssim\int_{0}^{2^{\rho s}\alpha}\left|\left\{|A(x)|>2^{\rho s-1}\alpha\right\}\right|\lambda d\lambda+\int_{2^{\rho s}\alpha}^{\infty}\left|\left\{|(A+B)(x)|>\lambda\right\}\right|\lambda d\lambda\\
			&\lesssim\|A\|_2^2+\mathscr{M}\alpha^2\sum\|b_Q\|_1\int_{2^{\rho s}\alpha}^{\infty}\frac{d\lambda}{\lambda^2}\\
			&\lesssim\mathscr{M}^2\alpha2^{-\frac{s}{8}}\sum\|b_Q\|_1.
		\end{align*}
		\textit{Estimate for $III$:} It is not difficult to see that an application of \eqref{claim:sublevel} implies that 
		\begin{align*}
			III&\lesssim\int_{0}^\infty\left|\left\{|(A+B)(x)|>\max\{\lambda,2^{\rho s-1}\alpha\}\right\}\right|\lambda d\lambda\lesssim\mathscr{M}^2\alpha2^{-\frac{s}{8}}\sum\|b_Q\|_1.
		\end{align*}
		We now prove \eqref{claim:sublevel}. We choose a smooth and positive function $\nu_0$ supported in $\{2^{-2}\leq|x|\leq2^2\}$ such that $\|\nu_0\|_\infty,\|\nu_0\|_1\leq10^d\mathscr{M}$ and $|H_0(x)|\leq\nu_0(x)$. We will prove the corresponding estimate for $\nu_j$ defined as $\nu_j(x)=2^{-jd}\nu_0(2^{-j}x)$ and the estimate \eqref{claim:sublevel} fo $H_j$ will follow as a consequence. We note that the estimate \eqref{claim:sublevel} follows from the following two estimates
		\begin{align}
			&\left\|\sum_{j\in\mathcal{I}}\nu_j*|B^{j-s}|\right\|_1\lesssim\mathscr{M}\sum\|b_Q\|_1,\label{est:claim1}\\
			&\left\|\sum_{j\in\mathcal{I}}\nu_j*|B^{j-s}|\right\|_{BMO}\lesssim\mathscr{M}\alpha.\label{est:claim2}
		\end{align}
		Indeed, by Chebyshev's inequality and interpolation, we have that
		\begin{align*}
			\left|\left\{x\in\R^d:|\sum_{j\in\mathcal{I}}H_j*B^{j-s}|>\lambda\alpha\right\}\right|&\lesssim\frac{1}{(\lambda\alpha)^3}\left\|\sum_{j\in\mathcal{I}}\nu_j*|B^{j-s}|\right\|_3^3\\
			&\lesssim\frac{\mathscr{M}}{(\lambda\alpha)^3}\left[\left(\sum\|b_Q\|_1\right)^\frac{1}{3}\alpha^\frac{2}{3}\right]^3\\
			&\lesssim\frac{\mathscr{M}}{\lambda^3\alpha}\sum\|b_Q\|_1.
		\end{align*}
		The estimate \eqref{est:claim1} is easy to verify. 
		
		We will provide the proof for inequality \eqref{est:claim2}. First, we observe that by the disjointness of the cubes in the collection $\mathcal{Q}$, for any cube $R\subset\R^d$ with side length $\ell(R)$, we have
		\begin{align}\label{est:suppcalculation}
			\sum_{2^j\leq\ell(R)}\|B^{j-s}\chi_R\|_1\lesssim\sum_{\substack{Q\in\mathcal{Q}:Q\cap R\neq\emptyset\\ \ell(Q)\leq\ell(R)}}\alpha|Q|\lesssim\alpha|R|.
		\end{align}
		Now, we fix a cube $R\subset\R^d$, we write
		\begin{align*}
			&\frac{1}{|R|}\int_R\left|\sum_{j\in\mathcal{I}}\nu_j*|B^{j-s}|(x)-\frac{1}{|R|}\int_R\sum_{j\in\mathcal{I}:2^j\geq\ell(R)}\nu_j*|B^{j-s}|(y)dy\right|dx\\
			\leq&\frac{1}{|R|}\int_R\left|\sum_{j\in\mathcal{I}:2^j\geq\ell(R)}\nu_j*|B^{j-s}|(x)-\frac{1}{|R|}\int_R\sum_{j\in\mathcal{I}:2^j\geq\ell(R)}\nu_j*|B^{j-s}|(y)dy\right|dx\\
			&+\frac{1}{|R|}\int_R\left|\sum_{j\in\mathcal{I}:2^j<\ell(R)}\nu_j*|B^{j-s}|(x)\right|dx\\
			=:&S_1+S_2.
		\end{align*}
		By mean value theorem and local support of the functions, we have
		\begin{align*}
			S_1&\lesssim\frac{1}{|R|}\int_R\sum_{j\in\mathcal{I}:2^j\geq\ell(R)}2^{-j}\ell(R)\|(\nabla\nu)_j\|_\infty\|B^{j-s}\chi_{B(c_R,2^{j+3})}|\|_1dx\\
			&\lesssim\sum_{j\in\mathcal{I}:2^j\geq\ell(R)}2^{-j}\ell(R)2^{-jd}\mathscr{M}\alpha|B(c_R,2^{j+3})|\\
			&\lesssim\mathscr{M}\alpha,
		\end{align*}
		where $C_R$ is the center of the cube $R$ and we used \eqref{est:suppcalculation} in the second step. 
		
		By the local support of the functions and \eqref{est:suppcalculation}, we have
		\begin{align*}
			S_2&\lesssim\frac{1}{|R|}\sum_{j\in\mathcal{I}:2^j<\ell(R)}\|\nu_j\|_1\|B^{j-s}\chi_{6R}\|_1\\
			&\lesssim\frac{\mathscr{M}}{|R|}\sum_{j\in\mathcal{I}:2^j<\ell(R)}\|B^{j-s}\chi_{6R}\|_1\\
			&\lesssim\frac{\mathscr{M}}{|R|}\alpha|R|=\mathscr{M}\alpha.
		\end{align*}
		Thus, \eqref{est:claim2} follows and this concludes the proof.
	\end{proof}
\section*{Acknowledgement}
Ankit Bhojak would like to thank Michael Lacey and Parasar Mohanty for various helpful conversations. Ankit Bhojak is supported by the Science and Engineering Research Board, Department of Science and Technology, Govt. of India, under the scheme National Post-Doctoral Fellowship, file no. PDF/2023/000708. Saurabh Shrivastava acknowledges the support by Anusandhan National Research Foundation, India under the project ANRF/ARG/2025/000940/MS. 
\bibliography{biblio}

\begin{bibdiv}
\begin{biblist}

\bib{BM2023}{article}{
      author={Bhojak, Ankit},
      author={Mohanty, Parasar},
       title={Weak type bounds for rough maximal singular integrals near {$L^1$}},
        date={2023},
        ISSN={0022-1236,1096-0783},
     journal={J. Funct. Anal.},
      volume={284},
      number={10},
       pages={Paper No. 109881, 27},
         url={https://doi.org/10.1016/j.jfa.2023.109881},
      review={\MR{4554742}},
}

\bib{Bourgain1989}{article}{
      author={Bourgain, Jean},
       title={Pointwise ergodic theorems for arithmetic sets},
        date={1989},
        ISSN={0073-8301,1618-1913},
     journal={Inst. Hautes \'Etudes Sci. Publ. Math.},
      number={69},
       pages={5\ndash 45},
         url={http://www.numdam.org/item?id=PMIHES_1989__69__5_0},
        note={With an appendix by the author, Harry Furstenberg, Yitzhak Katznelson and Donald S. Ornstein},
      review={\MR{1019960}},
}

\bib{CJRRW2000}{article}{
      author={Campbell, James~T.},
      author={Jones, Roger~L.},
      author={Reinhold, Karin},
      author={Wierdl, M\'at\'e},
       title={Oscillation and variation for the {H}ilbert transform},
        date={2000},
        ISSN={0012-7094,1547-7398},
     journal={Duke Math. J.},
      volume={105},
      number={1},
       pages={59\ndash 83},
         url={https://doi.org/10.1215/S0012-7094-00-10513-3},
      review={\MR{1788042}},
}

\bib{CJRW2003}{article}{
      author={Campbell, James~T.},
      author={Jones, Roger~L.},
      author={Reinhold, Karin},
      author={Wierdl, M\'at\'e},
       title={Oscillation and variation for singular integrals in higher dimensions},
        date={2003},
        ISSN={0002-9947,1088-6850},
     journal={Trans. Amer. Math. Soc.},
      volume={355},
      number={5},
       pages={2115\ndash 2137},
         url={https://doi.org/10.1090/S0002-9947-02-03189-6},
      review={\MR{1953540}},
}

\bib{CR1988}{article}{
      author={Christ, Michael},
      author={Rubio~de Francia, Jos\'{e}~Luis},
       title={Weak type {$(1,1)$} bounds for rough operators. {II}},
        date={1988},
        ISSN={0020-9910,1432-1297},
     journal={Invent. Math.},
      volume={93},
      number={1},
       pages={225\ndash 237},
         url={https://doi.org/10.1007/BF01393693},
      review={\MR{943929}},
}

\bib{CZ1956}{article}{
      author={Calder\'on, A.~P.},
      author={Zygmund, A.},
       title={On singular integrals},
        date={1956},
        ISSN={0002-9327,1080-6377},
     journal={Amer. J. Math.},
      volume={78},
       pages={289\ndash 309},
         url={https://doi.org/10.2307/2372517},
      review={\MR{84633}},
}

\bib{DHL2017}{article}{
      author={Ding, Yong},
      author={Hong, Guixiang},
      author={Liu, Honghai},
       title={Jump and variational inequalities for rough operators},
        date={2017},
        ISSN={1069-5869,1531-5851},
     journal={J. Fourier Anal. Appl.},
      volume={23},
      number={3},
       pages={679\ndash 711},
         url={https://doi.org/10.1007/s00041-016-9484-8},
      review={\MR{3649476}},
}

\bib{DOP2017}{article}{
      author={Do, Yen},
      author={Oberlin, Richard},
      author={Palsson, Eyvindur~A.},
       title={Variation-norm and fluctuation estimates for ergodic bilinear averages},
        date={2017},
        ISSN={0022-2518,1943-5258},
     journal={Indiana Univ. Math. J.},
      volume={66},
      number={1},
       pages={55\ndash 99},
         url={https://doi.org/10.1512/iumj.2017.66.5983},
      review={\MR{3623404}},
}

\bib{DR1986}{article}{
      author={Duoandikoetxea, Javier},
      author={Rubio~de Francia, Jos{\'e}~L.},
       title={Maximal and singular integral operators via {F}ourier transform estimates},
        date={1986},
        ISSN={0020-9910,1432-1297},
     journal={Invent. Math.},
      volume={84},
      number={3},
       pages={541\ndash 561},
         url={https://doi.org/10.1007/BF01388746},
      review={\MR{837527}},
}

\bib{HLP2013}{article}{
      author={Hyt\"onen, Tuomas~P.},
      author={Lacey, Michael~T.},
      author={P\'erez, Carlos},
       title={Sharp weighted bounds for the {$q$}-variation of singular integrals},
        date={2013},
        ISSN={0024-6093,1469-2120},
     journal={Bull. Lond. Math. Soc.},
      volume={45},
      number={3},
       pages={529\ndash 540},
         url={https://doi.org/10.1112/blms/bds114},
      review={\MR{3065022}},
}

\bib{Hofmann1988}{article}{
      author={Hofmann, Steve},
       title={Weak {$(1,1)$} boundedness of singular integrals with nonsmooth kernel},
        date={1988},
        ISSN={0002-9939,1088-6826},
     journal={Proc. Amer. Math. Soc.},
      volume={103},
      number={1},
       pages={260\ndash 264},
         url={https://doi.org/10.2307/2047563},
      review={\MR{938680}},
}

\bib{Honzik2020}{article}{
      author={Honz\'ik, Petr},
       title={An endpoint estimate for rough maximal singular integrals},
        date={2020},
        ISSN={1073-7928,1687-0247},
     journal={Int. Math. Res. Not. IMRN},
      number={19},
       pages={6120\ndash 6134},
         url={https://doi.org/10.1093/imrn/rny189},
      review={\MR{4165473}},
}

\bib{IKM2025}{article}{
      author={Iosevich, Alex},
      author={Krause, Ben},
      author={Mousavi, Hamed},
       title={An approach to endpoint problems in oscillatory singular integrals},
        date={2025},
        ISSN={0025-5874,1432-1823},
     journal={Math. Z.},
      volume={311},
      number={4},
       pages={Paper No. 66, 7},
         url={https://doi.org/10.1007/s00209-025-03859-8},
      review={\MR{4962612}},
}

\bib{JKRW1998}{article}{
      author={Jones, Roger~L.},
      author={Kaufman, Robert},
      author={Rosenblatt, Joseph~M.},
      author={Wierdl, M\'at\'e},
       title={Oscillation in ergodic theory},
        date={1998},
        ISSN={0143-3857,1469-4417},
     journal={Ergodic Theory Dynam. Systems},
      volume={18},
      number={4},
       pages={889\ndash 935},
         url={https://doi.org/10.1017/S0143385798108349},
      review={\MR{1645330}},
}

\bib{JRW2001}{article}{
      author={Jones, Roger~L.},
      author={Rosenblatt, Joseph~M.},
      author={Wierdl, M\'at\'e},
       title={Oscillation inequalities for rectangles},
        date={2001},
        ISSN={0002-9939,1088-6826},
     journal={Proc. Amer. Math. Soc.},
      volume={129},
      number={5},
       pages={1349\ndash 1358},
         url={https://doi.org/10.1090/S0002-9939-00-06032-9},
      review={\MR{1814160}},
}

\bib{JRW2003}{article}{
      author={Jones, Roger~L.},
      author={Rosenblatt, Joseph~M.},
      author={Wierdl, M\'at\'e},
       title={Oscillation in ergodic theory: higher dimensional results},
        date={2003},
        ISSN={0021-2172,1565-8511},
     journal={Israel J. Math.},
      volume={135},
       pages={1\ndash 27},
         url={https://doi.org/10.1007/BF02776048},
      review={\MR{1996394}},
}

\bib{JSW2008}{article}{
      author={Jones, Roger~L.},
      author={Seeger, Andreas},
      author={Wright, James},
       title={Strong variational and jump inequalities in harmonic analysis},
        date={2008},
        ISSN={0002-9947,1088-6850},
     journal={Trans. Amer. Math. Soc.},
      volume={360},
      number={12},
       pages={6711\ndash 6742},
         url={https://doi.org/10.1090/S0002-9947-08-04538-8},
      review={\MR{2434308}},
}

\bib{KL2018}{article}{
      author={Krause, Ben},
      author={Lacey, Michael~T.},
       title={Sparse bounds for maximal monomial oscillatory {H}ilbert transforms},
        date={2018},
        ISSN={0039-3223,1730-6337},
     journal={Studia Math.},
      volume={242},
      number={3},
       pages={217\ndash 229},
         url={https://doi.org/10.4064/sm8699-7-2017},
      review={\MR{3794333}},
}

\bib{KL2020}{article}{
      author={Krause, Ben},
      author={Lacey, Michael~T.},
       title={Sparse bounds for maximally truncated oscillatory singular integrals},
        date={2020},
        ISSN={0391-173X,2036-2145},
     journal={Ann. Sc. Norm. Super. Pisa Cl. Sci. (5)},
      volume={20},
      number={2},
       pages={415\ndash 435},
      review={\MR{4105908}},
}

\bib{Lep1976}{article}{
      author={L\'epingle, D.},
       title={La variation d'ordre {$p$} des semi-martingales},
        date={1976},
     journal={Z. Wahrscheinlichkeitstheorie und Verw. Gebiete},
      volume={36},
      number={4},
       pages={295\ndash 316},
         url={https://doi.org/10.1007/BF00532696},
      review={\MR{420837}},
}

\bib{Lai2025}{article}{
      author={Lai, Xudong},
       title={Weak $(1,1)$ estimate for maximal truncated rough singular integral operator},
        date={2025},
      eprint={2508.17737},
         url={https://arxiv.org/abs/2508.17737},
}

\bib{LL2012}{article}{
      author={Lewko, Allison},
      author={Lewko, Mark},
       title={Estimates for the square variation of partial sums of {F}ourier series and their rearrangements},
        date={2012},
        ISSN={0022-1236,1096-0783},
     journal={J. Funct. Anal.},
      volume={262},
      number={6},
       pages={2561\ndash 2607},
         url={https://doi.org/10.1016/j.jfa.2011.12.007},
      review={\MR{2885959}},
}

\bib{MST2017}{article}{
      author={Mirek, Mariusz},
      author={Stein, Elias~M.},
      author={Trojan, Bartosz},
       title={{$\ell^p(\Bbb Z^d) $}-estimates for discrete operators of {R}adon type: variational estimates},
        date={2017},
        ISSN={0020-9910,1432-1297},
     journal={Invent. Math.},
      volume={209},
      number={3},
       pages={665\ndash 748},
         url={https://doi.org/10.1007/s00222-017-0718-4},
      review={\MR{3681393}},
}

\bib{MST2019}{article}{
      author={Mirek, Mariusz},
      author={Stein, Elias~M.},
      author={Trojan, Bartosz},
       title={{$\ell^p(\Bbb{Z}^d)$}-estimates for discrete operators of {R}adon type: maximal functions and vector-valued estimates},
        date={2019},
        ISSN={0022-1236,1096-0783},
     journal={J. Funct. Anal.},
      volume={277},
      number={8},
       pages={2471\ndash 2521},
         url={https://doi.org/10.1016/j.jfa.2018.10.020},
      review={\MR{3990725}},
}

\bib{MSK2020}{article}{
      author={Mirek, Mariusz},
      author={Stein, Elias~M.},
      author={Zorin-Kranich, Pavel},
       title={A bootstrapping approach to jump inequalities and their applications},
        date={2020},
        ISSN={2157-5045,1948-206X},
     journal={Anal. PDE},
      volume={13},
      number={2},
       pages={527\ndash 558},
         url={https://doi.org/10.2140/apde.2020.13.527},
      review={\MR{4078235}},
}

\bib{MSZ2020}{article}{
      author={Mirek, Mariusz},
      author={Stein, Elias~M.},
      author={Zorin-Kranich, Pavel},
       title={Jump inequalities via real interpolation},
        date={2020},
        ISSN={0025-5831,1432-1807},
     journal={Math. Ann.},
      volume={376},
      number={1-2},
       pages={797\ndash 819},
         url={https://doi.org/10.1007/s00208-019-01889-2},
      review={\MR{4055178}},
}

\bib{MT2016}{article}{
      author={Mirek, Mariusz},
      author={Trojan, Bartosz},
       title={Discrete maximal functions in higher dimensions and applications to ergodic theory},
        date={2016},
        ISSN={0002-9327,1080-6377},
     journal={Amer. J. Math.},
      volume={138},
      number={6},
       pages={1495\ndash 1532},
         url={https://doi.org/10.1353/ajm.2016.0045},
      review={\MR{3595493}},
}

\bib{Seeger1996}{article}{
      author={Seeger, Andreas},
       title={Singular integral operators with rough convolution kernels},
        date={1996},
        ISSN={0894-0347,1088-6834},
     journal={J. Amer. Math. Soc.},
      volume={9},
      number={1},
       pages={95\ndash 105},
         url={https://doi.org/10.1090/S0894-0347-96-00185-3},
      review={\MR{1317232}},
}

\end{biblist}
\end{bibdiv}
\end{document}